\documentclass{article}
\usepackage{amsmath, amssymb, graphics}

\usepackage{amsthm}
\usepackage[dvipdfmx]{graphicx}
\usepackage{color}
\usepackage{natbib}
\usepackage{amscd} 
\usepackage{tikz-cd}
\usepackage{here}
\usepackage{pb-diagram}

\usepackage{threeparttable}
\usepackage{caption}
\newtheoremstyle{mystyle}
    {}
    {}
    {\itshape}
    {}
    {\bfseries}
    {}
    { }
    {}
  \theoremstyle{mystyle}

\newtheorem{thm}{Theorem}[section]
\newtheorem{df}{Definition}[section]

\newtheorem{prp}{Proposition}[section]
\newtheorem{lem}{Lemma}[section]
\newtheorem{cor}{Corollary}[section]

\newcommand{\mathsym}[1]{{}}
\newcommand{\unicode}[1]{{}}

\makeatletter
    
    \@addtoreset{equation}{section}
  \makeatother

\begin{document}
\title{Construction of smooth rhythms through a monotone invariant measure}
\author{Fumio HAZAMA\\Tokyo Denki University\\Hatoyama, Hiki-Gun, Saitama JAPAN\\
e-mail address:hazama@mail.dendai.ac.jp}
\date{\today}
\maketitle
\thispagestyle{empty}

\begin{abstract}
\noindent
The present article introduces the notion of quasi-smoothness of marked rhythms through a certain transformation $Ref$, called reformation map. A marked rhythm consists of a rhythm together with a marker, and the map $Ref$ modifies the marked onset of the rhythm. It is shown that an iteration of the map $Ref$ transforms an arbitrary marked rhythm into a quasi-smooth one. A numerical criterion for a marked rhythm to be quasi-smooth is given in terms of the difference of its rhythm part. Through this criterion, the rhythm part of any quasi-smooth marked rhythm is shown to be smooth.
\end{abstract}

\noindent
$\mathbf{Keywords}$: marked rhythm, reformation map, deformation map, quasi-smoothness, finite dynamical system, monotone invariant measure.\\

\section{Introduction}
The main purpose of this paper is to propose a new method of construction of smooth rhythms. The notion of smoothness of a rhythm is introduced in (Hazama 2022), and is shown to give us a larger category of rhythms than that of {\it maximally even} ones, which have been widely studied (Clough et a. 2000; Demaine 2009; Toussaint 2013). In (Hazama 2022), a self-map $Rav$ on the space $\mathbf{R}_N^n$ of rhythms of $N$ beats with $n$ onsets is introduced, and a rhythm is defined to be smooth if it is periodic under $Rav$. One of the main results of (Hazama 2022) gives us a numerical criterion for a rhythm to be smooth in terms of its width. The proof of the validity of the criterion, given there, is quite involved however, and the author of the paper expresses his hope about a possible simplification of its original proof. Our main result realizes his hope by introducing a self-map $Ref$, called {\it \underline{ref}ormation map}, on the space $m\mathbf{R}_N^n=\mathbf{Z}_n\times \mathbf{R}_N^n$ of marked rhythms. The first coordinate $k$ of a marked rhythm $(k,\mbox{\boldmath $a$})$ indicates which entry of the rhythm $\mbox{\boldmath $a$}$ is modified by $Ref$. Since $Ref$ does not change any other entries, the dynamical behavior of $Ref$ will be seen to be much simpler than that of $Rav$, which changes all the entries at the same time. Our main result shows that, if a marked rhythm is periodic under $Ref$, then its rhythm part is smooth in the sense of (Hazama 2022). 

The plan of this paper is as follows. Section one introduces our main tool {\it Ref} on $m\mathbf{R}_N^n$ for our study. The map leads us to a finite dynamical system $(m\mathbf{R}_N^n, Ref)$, which we denote by $\mathbf{Ref}_N^n$. The notion of {\it quasi-smoothness} of a marked rhythm is defined by the property that it is periodic under $Ref$. We investigate the dynamics of $Ref$ through its descent to a quotient dynamical system $\mathbf{Def}_N^n=(m\mathbf{D}_N^n,Def)$ by an automorphism of $\mathbf{Ref}_N^n$. The compatibility of $Ref$ and $Def$ is proved at the end of this section. Section two axiomatizes a crucial point which emerges in our treatment of the dynamical system $\mathbf{Ref}_N^n$ together with its quotient system $\mathbf{Def}_N^n$. We introduce the notion of ``{\it monotone invariant measure}'' on an arbitrary finite dynamical system $\mathbf{F}=(X,F)$ with an automorphism $\varphi$. It is shown to descend to the quotient dynamical system $\overline{\mathbf{F}}=(\overline{X},\overline{F})$ by the action of $\varphi$. The measure enables one to relate the periodic behavior of the self-map $F$ on $X$ with that of the self-map $\overline{F}$ on $\overline{X}$. Section three constructs a monotone invariant measure $\mu_{m\mathbf{R}}$ on $\mathbf{Ref}_N^n$. This is modeled after one of the measures proposed and investigated in (Rota 2001). A general result in the previous section allows $\mu_m{\mathbf{R}}$ to descend to a measure $\mu_{m\mathbf{D}}$ on $\mathbf{Def}_N^n$, and enables us to relate the periodicity of $Def$ and that of $Ref$. Section four introduces the notion of $\mu_{m\mathbf{D}}$-{\it stability}, which is shown to be a necessary condition for the periodicity. The final step toward our goal is provided by the fact that $\mu_{m\mathbf{D}}$-stability of $\mbox{\boldmath $D$}\in m\mathbf{D}_N^n$ implies that the width of $\mbox{\boldmath $D$}$ is less than or equal to 1. Thus we arrive at our main theorem which shows the equivalence of the quasi-smoothness and the smoothness.

\section{Reformation map for marked rhythms}
In this section we fix some notation and terminology, and introduce our main subject, the {\it reformation map}.

\subsection{Basic setup}
For any integer $m\geq 3$, let $\mathbf{Z}_m=\{0,1,\cdots, m-1\}$,  and let $R_m:\mathbf{Z}\rightarrow \mathbf{Z}_m$ denote the function which maps an integer to its least nonnegative remainder modulo $m$. The usual addition and subtraction operations on $\mathbf{Z}/m\mathbf{Z}$ are translated to the ones on the set $\mathbf{Z}_m$, which we denote by ``$+_m$" and ``$-_m$" respectively. Namely we define
\begin{eqnarray*}
a+_mb&=&R_m(a+b),\\
a-_mb&=&R_m(a-b),
\end{eqnarray*}
for any pair $a,b\in\mathbf{Z}_m$. For any pair $a,b$ of distinct elements in $\mathbf{Z}_m$, we define the {\it close interval} by $[a,b]_m=\{a,a+_m1,\cdots,b\}\subset \mathbf{Z}_m$. \\

Let $N$ be an integer $\geq 3$, and let $n$ be an integer with $3\leq n\leq N$. Throughout the paper the integers $N,n$ are assumed to satisfy these conditions. An $n$-tuple $\mbox{\boldmath $a$}=(a_0,a_1,\cdots, a_{n-1})\in\mathbf{Z}_N^n$ of mutually distinct elements of $\mathbf{Z}_N$ is a {\it rhythm} of length $N$ with $n$ onsets, if it satisfies the condition
\begin{eqnarray}
\sum_{k=0}^{n-1}(a_{k}-_Na_{k-_n1})=N.
\end{eqnarray}
Here the summation on the left hand side means the usual addition on $\mathbf{Z}$. We denote by $\mathbf{R}_N^n$ the set of rhythms of length $N$ with $n$ onsets. A {\it marked rhythm} is the pair $(k,\mbox{\boldmath $a$})$ of an element of $k\in \mathbf{Z}_n$, called the {\it marker}, and a rhythm $\mbox{\boldmath $a$}\in\mathbf{R}_N^n$, called the {\it rhythm part}. The $k$-th coordinate $a_k$ is called the {\it marked entry} of the marked rhythm $(k,\mbox{\boldmath $a$})$. The set of marked rhythm is denoted by $m\mathbf{R}_N^n$:
\begin{eqnarray*}
m\mathbf{R}_N^n=\mathbf{Z}_n\times \mathbf{R}_N^n.
\end{eqnarray*}
Our main ingredient $Ref$ of this papr will be defined as a self-map on $m\mathbf{R}_N^n$. \\

For our investigation of rhythms and marked rhythms, it is indispensable to consider their differences, whose effectiveness is also fully illustrated in the comprehensive book (Toussaint 2013). For any rhythm $\mbox{\boldmath $a$}$, its {\it difference} $d(\mbox{\boldmath $a$})$ is defined by
\begin{eqnarray*}
d(\mbox{\boldmath $a$})=(a_0-_Na_{n-1},a_1-_Na_0,\cdots,a_{n-1}-_Na_{n-2})
\end{eqnarray*}
Since a rhythm consists of mutually distinct elements of $\mathbf{Z}_N$, it follows from (1.1) that the difference belongs to the set
\begin{eqnarray*}
\mathbf{D}_N^n=\{(d_0,d_1,\cdots,d_{n-1})\in\mathbf{Z}_N^*|\sum_{k=0}^{n-1}d_k=N\},
\end{eqnarray*}
where $\mathbf{Z}_N^*=\mathbf{Z}_N\setminus \{0\}$. For any marked rhythm $(k,\mbox{\boldmath $a$})\in m\mathbf{R}_N^n$, the {\it difference} $\Delta(k,\mbox{\boldmath $a$})$ is defined by
\begin{eqnarray*}
\Delta(k,\mbox{\boldmath $a$})=(k,d(\mbox{\boldmath $a$})).
\end{eqnarray*}
Hence it belongs to
\begin{eqnarray*}
m\mathbf{D}_N^n=\mathbf{Z}_n\times \mathbf{D}_N^n.
\end{eqnarray*}
Its elements are called {\it marked differences}. The first component is called the {\it marker} and the second is the {\it difference part}.\\

Let $tr$ denote the bijective self-map on $\mathbf{R}_N^n$ defined by
\begin{eqnarray}
tr(\mbox{\boldmath $a$})=(a_0+_N1,\cdots,a_{n-1}+_N1)
\end{eqnarray}
for any $\mbox{\boldmath $a$}=(a_0,\cdots,a_{n-1})\in \mathbf{R}_N^n$, and let $\sim_{tr}$ denote the equivalence relation on $\mathbf{R}_N^n$ generated by the bijection $tr$. It is not hard to see that the space $\mathbf{D}_N^n$ is identified with the quotient space $\mathbf{R}_N^n/\sim_{tr}$ of $\mathbf{R}_N^n$ by the equivalence relation $\sim_{tr}$. Similarly, if we define the bijective self-map $mtr$ on $m\mathbf{R}_N^n$ by
\begin{eqnarray}
mtr(k,\mbox{\boldmath $a$})=(k,tr(\mbox{\boldmath $a$})),
\end{eqnarray}
and denote by $\sim_{mtr}$ the equivalence relation generated by the bijection $mtr$, then we see that the space $m\mathbf{D}_N^n$ is identified with the quotient space $m\mathbf{R}_N^n/\sim_{mtr}$ of $m\mathbf{R}_N^n$ by the equivalence relation $\sim_{mtr}$. 
\\

\subsection{Definition of the reformation map $Ref$}
In this section, we define a self map ``{\it Ref}'', called {\it reformation map}, on $m\mathbf{R}_N^n$. It is based on the discrete average $rav:\mathbf{Z}_N\times \mathbf{Z}_N\rightarrow \mathbf{Z}_N$ defined as follows.

\begin{df}
For any pair $(a,b)\in\mathbf{Z}_N^2$, its {\rm discrete average} $rav_N(a,b)\in\mathbf{Z}_N$ is defined by
\begin{eqnarray*}
rav_N(a,b)=a+_N\left\lfloor\frac{b-_Na}{2}\right\rfloor.
\end{eqnarray*}
\end{df}

\noindent
Note that the discrete average is compatible with the addition by one modulo $N$, namely we have
\begin{eqnarray}
rav_N(a+_N1,b+_N1)=rav_N(a,b)+_N1.
\end{eqnarray}
Although this appears to be trivial, the whole our study in this paper is built on this simple equality. We define a self map $ref:\mathbf{R}_N^3\rightarrow \mathbf{R}_N^3$ as follows.

\begin{df}
For any $(a,b,c)\in\mathbf{R}_N^3$, we put
\begin{eqnarray}
ref(a,b,c)=(a,rav_N(a,c),c).
\end{eqnarray}
\end{df}

\noindent
Note that the map $ref$ transforms only the middle entry into the discrete average of the first and the third entries. It follows from (1.4) that
\begin{eqnarray}
ref\circ tr=tr\circ ref.
\end{eqnarray}
\noindent
Based on this map, we define the {\it reformation map}, denoted by $Ref$, on the space $m\mathbf{R}_N^n$ as follows.

\begin{df}
For any $(k,\mbox{\boldmath $a$})\in m\mathbf{R}_N^n$ with $\mbox{\boldmath $a$}=(a_0,a_1,\cdots,a_{n-1})$, we define $Ref(k,\mbox{\boldmath $a$})$ to be $(k+_n1,\mbox{\boldmath $b$})$, where $\mbox{\boldmath $b$}=(b_0,b_1,\cdots,b_{n-1})$ is given in the following way: Let $(c_{k-_n1},c_k,c_{k+_n1})=ref(a_{k-_n1},a_k,a_{k+_n1})$, and define
\begin{eqnarray*}
b_i=
\begin{cases}
c_{k-_n1}, & \mbox{ if }i=k-_n1,\\
c_k, & \mbox{ if }i=k,\\
c_{k+_n1}, & \mbox{ if }i=k+_n1,\\
a_i, & \mbox{ if }i\notin \{k-_n1,k,k+_n1\}\\
\end{cases}
\end{eqnarray*}
\end{df}

We see below through a few examples how the reformation map acts on marked rhythms. In order to visualize marked rhythms, we employ their {\it circle graphs}, whose construction follows. When $\mbox{\boldmath $A$}=(k, \mbox{\boldmath $a$})=(k,(a_0,\cdots, a_{n-1}))$ is an element of $m\mathbf{R}_N^n$, we draw a small disk at $\chi_N(a_j)\in\mathbf{C}$ for any $j\in\mathbf{Z}_n$, where $\chi_N:\mathbf{Z}_N\rightarrow \mathbf{C}$ is defined by $\chi_N(a)=e^{2\pi ia/N}$, and connect the $n$ disks by a polygonal line. The marked entry is enclosed in a circle. The numbers $m=0, 1, \cdots, N-1$, surrounding the unit circle, indicate the spots where the nodes $\chi_N(m)$ locate.
\\\\
\noindent
\textbf{Example 1} For the marked rhythm $\mbox{\boldmath $A$}=(0, (0,1,2))\in m\mathbf{R}_8^3$, the marked entry ``$0$'' is moved by $Ref$ to the discrete average ``$5$'' of ``$2$'' and ``$1$'', and the marker ``$0$'' goes to the next ``$1$'': \\

\begin{figure}[htbp]
\begin{center}
\includegraphics[scale=0.7]{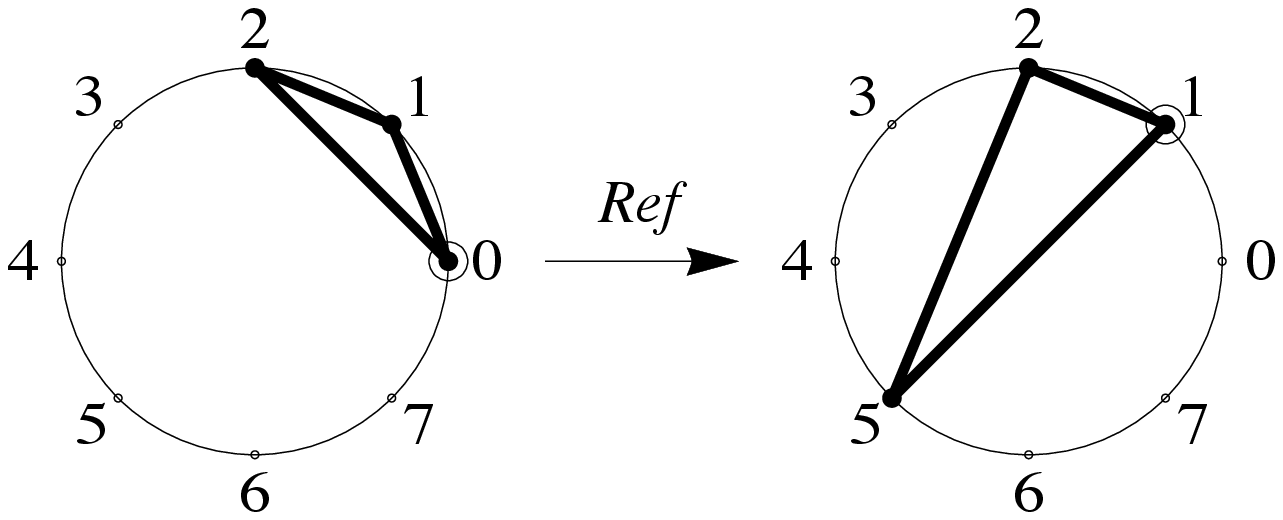}
\caption{$\mbox{\boldmath $A$}=(0, (0,1,2))$ and $Ref(\mbox{\boldmath $A$})=\mbox{\boldmath $A$}^{(1)}=(1, (5,1,2))$}
\end{center}
\end{figure}

\noindent
By iterating the application of $Ref$, the marked rhythm $\mbox{\boldmath $A$}^{(1)}=(1, (5,1,2))$ on the right of Figure 1 is mapped to $\mbox{\boldmath $A$}^{(2)}=(2, (5,7,2))$:

\begin{figure}[htbp]
\begin{center}
\includegraphics[scale=0.7]{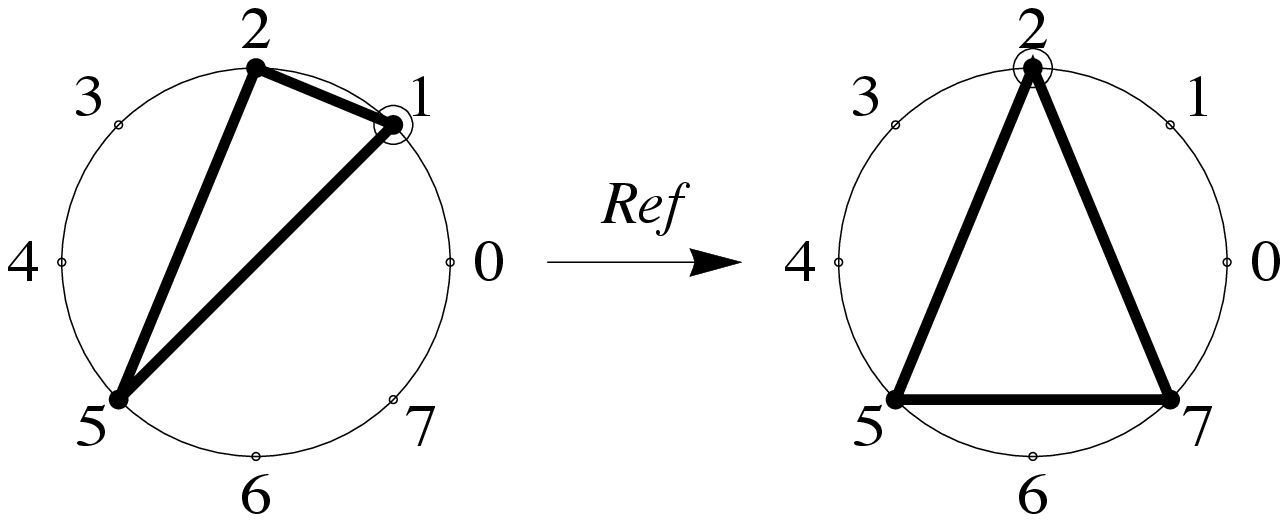}
\caption{$\mbox{\boldmath $A$}^{(1)}=(1, (5,1,2))$ and $Ref(\mbox{\boldmath $A$}^{(1)})=\mbox{\boldmath $A$}^{(2)}=(2, (5,7,2))$}
\end{center}
\end{figure}

\noindent
We see that the rhythm part of the marked rhythm $\mbox{\boldmath $A$}^{(2)}$ is smooth, since its width is equal to one. The main purpose of this paper is to show that this is a general phenomenon. Namely, we can make the rhythm part of an arbitrary marked rhythm smooth by iterated applications of $Ref$.\\\\

The equality (1.6) implies that the compatibility relation
\begin{eqnarray}
Ref\circ mtr=mtr\circ Ref
\end{eqnarray}
holds. Hence the self-map $Ref$ on $m\mathbf{R}_N^n$ descends to $m\mathbf{D}_N^n=m\mathbf{R}_N^n/\sim_{mtr}$. A concrete expression of the descended map is given in the next subsection.

\subsection{Descent of $Ref$ to $m\mathbf{D}_N^n$}
We introduce a self map $Def$, called {\it \underline{def}ormation map}, on $m\mathbf{D}_N^n$, and show that it is compatible with the reformation map $Ref$ through the difference map $\Delta$.

\begin{df}
For any $(k,\mbox{\boldmath $d$})\in m\mathbf{D}_N^n$ with $\mbox{\boldmath $d$}=(d_0,d_1,\cdots,d_{n-1})$, we define $Def(k,\mbox{\boldmath $d$})$ to be $(k+_n1,\mbox{\boldmath $e$})$, where $\mbox{\boldmath $e$}=(e_0,e_1,\cdots,e_{n-1})$ is given in the following way: Let 
\begin{eqnarray*}
(f_k,f_{k+_n1})=(\left\lfloor\frac{d_k+_Nd_{k+_n1}}{2}\right\rfloor,\left\lceil\frac{d_k+_Nd_{k+_n1}}{2}\right\rceil),
\end{eqnarray*}
and define
\begin{eqnarray*}
e_i=
\begin{cases}
f_k, & \mbox{ if }i=k,\\
f_{k+_n1}, & \mbox{ if }i=k+_n1,\\
d_i, & \mbox{ if }i\notin \{k,k+_n1\}\\
\end{cases}
\end{eqnarray*}
\end{df}

\noindent

\begin{prp}
We have the following commutative diagram:
\begin{eqnarray}
  \begin{CD}
     m\mathbf{R}_N^n @>{\Delta}>> m\mathbf{D}_N^n \\
  @V{Ref}VV @V{Def}VV \\
    m \mathbf{R}_N^n @>{\Delta}>> m\mathbf{D}_N^n \\
  \end{CD}
\end{eqnarray}
\end{prp}

\begin{proof}
Let $(k,\mbox{\boldmath $a$})=(k,(a_0,\cdots,a_{n-1}))\in m\mathbf{R}_N^n$ be an arbitrary marked rhythm, and let $Ref(k,\mbox{\boldmath $a$})=(k+_n1,\mbox{\boldmath $b$})$. It follows from Definition 1.3 that $b_i=a_i$ for any $i\in\mathbf{Z}_n\setminus\{k\}$. Furthermore, if we put $\Delta(k,\mbox{\boldmath $a$})=(k,\mbox{\boldmath $d$})$ and $Def(k,\mbox{\boldmath $d$})=(k+_n1,\mbox{\boldmath $e$})$, then it follows from Definition 1.4 that $e_i=d_i$ for any $i\in\mathbf{Z}_n\setminus\{k,k+_n1\}$.Therefore, for the proof of the commutativity, we have only to focus on the behavior of the subtriple $(a_{k-_n1},a_k,a_{k+_n1})$ of the rhythm part $\mbox{\boldmath $a$}$. It follows that we are reduced to proving the assertion for the case when $n=3$. Let $(1,(a_0,a_1,a_2))\in m\mathbf{R}_N^3$ be an arbitrary marked rhythm with marker equal to 1. It is mapped by $\Delta\circ Ref$ to
\begin{eqnarray}
&&\Delta(Ref(1,(a_0,a_1,a_2)))\nonumber\\
&&=\Delta(2,(a_0,rav_N(a_0,a_2),a_2))\nonumber\\
&&=(2,(a_0-_Na_2,rav_N(a_0,a_2)-_Na_0,a_2-_Nrav_N(a_0,a_2)))\nonumber\\
&&=(2,(a_0-_Na_2,\left\lfloor\frac{a_2-_Na_0}{2}\right\rfloor,(a_2-_Na_0)-_N\left\lfloor\frac{a_2-_Na_0}{2}\right\rfloor))\nonumber\\
&&=(2,(a_0-_Na_2,\left\lfloor\frac{a_2-_Na_0}{2}\right\rfloor,\left\lceil\frac{a_2-_Na_0}{2}\right\rceil)),
\end{eqnarray}
the last equality being a consequence of the equality $x=\left\lfloor\frac{x}{2}\right\rfloor+_N\left\lceil\frac{x}{2}\right\rceil$, which holds for any $x\in\mathbf{Z}_N$.
On the other hand, the value of $Def\circ\Delta$ at $(1,(a_0,a_1,a_2))$ is equal to
\begin{eqnarray}
&&Def(\Delta(1,(a_0,a_1,a_2)))\nonumber\\
&&=Def(1,(a_0-_Na_2,a_1-_Na_0,a_2-_Na_1))\nonumber\\
&&=(1,(a_0-_Na_2,\left\lfloor\frac{(a_1-_Na_0)+_N(a_2-_Na_1)}{2}\right\rfloor,\nonumber\\
&&\hspace{50mm}\left\lceil\frac{(a_1-_Na_0)+_N(a_2-_Na_1)}{2}\right\rceil)\nonumber\\
&&=(1,(a_0-_Na_2,\left\lfloor\frac{a_2-_Na_0}{2}\right\rfloor,\left\lceil\frac{a_2-_Na_0}{2}\right\rceil)).
\end{eqnarray}
Since the rightmost sides of (1.9) and (1.10) coincide, we see that the two maps $Def\circ\Delta$ and $\Delta\circ Def$ give one and the same value at $(1,(a_0,a_1,a_2))$. Since even if the marker is not equal to 1, a similar proof can be given, this completes the proof.
\end{proof}

\subsection{Finite dynamical systems $\mathbf{Ref}_N^n$ and $\mathbf{Def}_N^n$}
We introduce two finite dynamical systems employing the maps $Ref$ and $Def$.

\begin{df}
{\rm (1)} Let $\mathbf{Ref}_N^n$ denote the dynamical system $(m\mathbf{R}_N^n,Ref)$.\\
{\rm (2)} Let $\mathbf{Def}_N^n$ denote the dynamical system $(m\mathbf{D}_N^n,Def)$.\\
\end{df}

\noindent
The notion of {\it quasi-smoothness} is introduced through the periodicity of the self map $Ref$. For any finite dynamical system $\mathbf{F}=(X,F)$, let $Per(\mathbf{F})$ denote the set of periodic points:
\begin{eqnarray*}
Per(\mathbf{F})=\{x\in X|\mbox{ There exists a positive integer }m\mbox{ such that }F^m(x)=x\}.
\end{eqnarray*}

\begin{df}
A marked rhythm $\mbox{\boldmath $A$}\in m\mathbf{R}_N^n$ is {\rm quasi-smooth} if $\mbox{\boldmath $A$}\in Per(\mathbf{Ref}_N^n)$. A rhythm $\mbox{\boldmath $a$}\in\mathbf{R}_N^n$ is {\rm quasi-smooth} if there exists a $k\in\mathbf{Z}_n$ such that $\iota_k(\mbox{\boldmath $a$})$ is quasi-smooth as a marked rhythm, where $\iota_k:\mathbf{R}_N^n\rightarrow m\mathbf{R}_N^n$ denotes an inclusion defined by $\iota_k(\mbox{\boldmath $a$})=(k,\mbox{\boldmath $a$})\in m\mathbf{R}_N^n$.
\end{df}

The main theorem in this article is summarized as follows:\\

\hspace{5mm}{\it A rhythm is smooth if and only if it is quasi-smooth.}\\

\noindent
Thus we obtain another iterative method of construction for smooth rhythms, which is much simpler than the one employed in (Hazama 2022). Furthermore, as the reader will see in Section four, the long-term behavior of the self map $Ref$ is far more easier for us to understand than that of $Rav$.

\section{Monotone invariant measure}
We axiomatize the method of construction of quasi-smooth rhythms. This will enable us to prove the validity not only of our iterative method, but also of other possible methods which produce certain target objects with a desired property in a given dynamical system.\\

\subsection{Automorphism, invariance, and monotonicity}

\begin{df}
Let $\mathbf{F}=(X,F)$ be a finite dynamical system.\\
{\rm (1)} A bijection $\varphi:X\rightarrow X$ is called an {\rm automorphism} of $\mathbf{F}$ if it commutes with $F$, namely the equality
\begin{eqnarray}
F\circ \varphi=\varphi\circ F
\end{eqnarray}
holds. We denote by ``$\sim_{\varphi}$'' the equivalence relation on $X$ generated by $\varphi$, and put $\overline{X}=X/\sim_{\varphi}$, the quotient space of $X$ by $\sim_{\varphi}$.\\
{\rm (2)} A real valued function $\mu$ on $X$ is called an {\rm invariant measure} on $\mathbf{F}$ with respect to an automorphism $\varphi$ of $\mathbf{F}$, if
\begin{eqnarray}
\mu(\varphi(x))=\mu(x)
\end{eqnarray}
holds for any $x\in X$. \\
{\rm (3)} A real valued function $\mu$ on $X$ is called a {\rm monotone measure} on $\mathbf{F}$, if
\begin{eqnarray}
\mu(F(x))\geq \mu (x)
\end{eqnarray}
holds for any $x\in X$. 
\end{df}

\noindent
Under the conditions (2.1) and (2.2), the self-map $F$ as well as the function $\mu$ on $X$ descends to the quotient space $\overline{X}$. We record this fact for later use.

\begin{prp}
{\rm (1)} There exists a unique self map $\overline{F}$ on the quotient $\overline{X}$ such that the following diagram commutes:
\begin{eqnarray}
  \begin{CD}
     X @>{F}>> X \\
  @V{\pi}VV @V {\pi}VV \\
    \overline{X} @>{\overline{F}}>> \overline{X} \\
  \end{CD}
\end{eqnarray}
where $\pi:X\rightarrow\overline{X}$ denotes the quotient map. \\
{\rm (2)} The function $\mu$ descends to a unique function $\overline{\mu}$ on the quotient $\overline{X}$ such that the following diagram commutes:
\begin{eqnarray}
  \begin{CD}
     X @>{\mu}>> \mathbf{R} \\
  @V{\pi}VV @V {id_{\mathbf{R}}}VV \\
    \overline{X} @>{\overline{\mu}}>> \mathbf{R} \\
  \end{CD}
\end{eqnarray}
\end{prp}

The following proposition relates the monotonicity of $\mu$ with that of $\overline{\mu}$.

\begin{prp}
Suppose that $\mu$ is an invariant measure on $\mathbf{F}$ with respect to an automorphism $\varphi$ of $\mathbf{F}$. Let $\overline{F}$ and $\overline{\mu}$ denote the induced maps defined as above. Then $\mu$ is a monotone measure on $\mathbf{F}$ if and only if $\overline{\mu}$ is a monotone measure on $\overline{\mathbf{F}}$.
\end{prp}

\begin{proof}
For any $\overline{x}\in\overline{X}$, take an arbitrary $x\in \pi^{-1}(\{\overline{x}\})$. Then we have
\begin{eqnarray*}
\overline{\mu}(\overline{F}(\overline{x}))&=&\overline{\mu}(\overline{F}(\pi(x)))\\
&=&\overline{\mu}(\pi(F(x)))\hspace{15mm}(\mbox{by (2.4)})\\
&=&\mu(F(x))\hspace{20mm}(\mbox{by (2.5)})
\end{eqnarray*}
Since $\overline{\mu}(\overline{x})=\overline{\mu}(\pi(x))=\mu(x)$, the above computation implies that the inequality $\overline{\mu}(\overline{F}(\overline{x}))\geq \overline{\mu}(\overline{x})$ holds if and only if $\mu(F(x))\geq \mu (x)$. This completes the proof.
\end{proof}

For any dynamical system $(X,F)$, let $Per(X,F)$ denote the set of periodic points, namely
\begin{eqnarray*}
Per(X,F)=\{x\in X|\mbox{ there exists a positive integer $m$ such that $F^m(x)=x$}\}.
\end{eqnarray*}

\noindent
A monotone measure provides us with a useful necessary condition for an element $x\in X$ to be periodic.

\begin{prp}
Let $\mu$ be a monotone measure on $(X,F)$. If $x\in Per(X,F)$, then we have
\begin{eqnarray}
\mu(F(x))=\mu(x).
\end{eqnarray}
Furthermore, for any positive integer $k$, we have
\begin{eqnarray}
\mu(F^{k+1}(x))=\mu(F^k(x)).
\end{eqnarray}
\end{prp}

\begin{proof}
Contrary to (2.6), suppose that we have
\begin{eqnarray*}
\mu(F(x))>\mu(x).
\end{eqnarray*}
Let $m$ be a period of $x$ with respect to $F$. By the monotonicity (2.3) we have the following chain of inequalities.
\begin{eqnarray*}
\mu(x)<\mu(F(x))\leq \mu(F^2(x))\leq \mu (F^m(x))=\mu(x).
\end{eqnarray*}
This contradiction shows that we must have $\mu(F(x))=\mu(x)$. The equality (2.7) follows from (2.6), since $F^k(x)$ is also a periodic point. This completes the proof.
\end{proof}

In view of the importance of the condition (2.6) and (2.7) for the study of the periodic points of a dynamical system, we introduce the following.

\begin{df}
${\rm (1)}$ A point $x\in X$ is said to be $(\mu,F)$-{\rm invariant} if the equality $\mu(F(x))=\mu(x)$ holds.\\
${\rm (2)}$ A point $x\in X$ is said to be $\mu$-{\rm stable} if it is $(\mu,F^{\ell})$-invariant for any positive integer $\ell$. The set of $\mu$-stable points is denoted by $St(\mu)$.\\
\end{df}

\noindent
Through this terminology, Proposition 2.3 can be rephrased as follows:

\begin{cor}
Let $\mu$ be a monotone measure on $(X,F)$. Then we have
\begin{eqnarray*}
Per(X,F)\subset St(\mu).
\end{eqnarray*}
\end{cor}

\noindent
In other words, the $\mu$-stability is a necessary condition for the periodicity.\\

Under the existence of an automorphism of a dynamical system $(X,F)$, we can relate the periodic points of  $(X,F)$ with the periodic points of the quotient system $(\overline{X},\overline{F})$.  

\begin{prp}
Let $\varphi$ be an automorphism of a finite dynamical system $(X,F)$. For an element $x\in X$ to belong $Per(X,F)$, it is necessary and sufficient that $\pi(x)\in Per(\overline{X},\overline{F})$.
\end{prp}

\begin{proof}
(Necessity) Let $x\in Per(X,F)$ so that there exists a positive integer $m$ such that 
\begin{eqnarray}
F^m(x)=x.
\end{eqnarray}
The left hand side is mapped through $\pi$ to 
\begin{eqnarray}
\pi(F^m(x))=\overline{F}^m(\pi(x))
\end{eqnarray}
 by the very definition of $\overline{F}$, namely by the equality $\pi\circ F=\overline{F}\circ \pi$. Since $\pi$ maps the right hand side to $\pi(x)$, this shows the equality $\overline{F}^m(\pi(x))=\pi(x)$. Hence $\pi(x)\in Per(\overline{X},\overline{F})$. \\
(Sufficiency). 
Let $\pi(x)\in Per(\overline{X},\overline{F})$ so that there exists a positive integer $m$ such that 
\begin{eqnarray}
\overline{F}^m(\pi(x))=\pi(x).
\end{eqnarray}
The left hand side is equal to $\overline{F}^m(\pi(x))=\pi(F^m(x))$ by (2.8). Hence, together with (2.9), it implies that
\begin{eqnarray*}
F^m(x)\sim_{\varphi} x.
\end{eqnarray*}
It follows from the definition of the relation $\sim_{\varphi}$ that there exists a positive integer $k$ such that
\begin{eqnarray}
F^m(x)=\varphi^k(x).
\end{eqnarray}
Since $\varphi$ commutes with $F$, applying $F^m$ on both sides, we see that
\begin{eqnarray*}
F^{2m}(x)=F^m(\varphi^k(x))=\varphi^k(F^m(x))=\varphi^{2k}(x),
\end{eqnarray*}
the last equality being a consequence of (2.10). Repeating the process, we see that, for any positive integer $p$, we have
\begin{eqnarray*}
F^{pm}(x)=\varphi^{pk}(x).
\end{eqnarray*}
Therefore, when $p$ is the order of the automorphism $\varphi$, we have
\begin{eqnarray*}
F^{pm}(x)=\varphi^{pk}(x)=x,
\end{eqnarray*}
which shows that $x\in Per(X,F)$. This completes the proof.
\end{proof}

\section{Monotone invariant measure $\mu_{m\mathbf{R}}$ on $m\mathbf{Ref}_N^n$}
We introduce a $\mathbf{Z}$-valued function $\mu_{m\mathbf{R}}$ on the dynamical system $\mathbf{Ref}_N^n=(m\mathbf{R}_N^n, Ref)$. Recall that the self-map $mtr$ on $m\mathbf{R}_N^n$ is an automorphism of $\mathbf{Ref}_N^n$ in the sense of Definition 2.1 (1), since we have the equality (1.7). We will show that $\mu_{m\mathbf{R}}$ is a monotone invariant measure with respect to the automorphism $mtr$.

\begin{df}
For any $\mbox{\boldmath $a$}=(a_0,\cdots,a_{n-1})\in\mathbf{R}_N^n$, let $\mu_{\mathbf{R}}:\mathbf{R}_N^n\rightarrow\mathbf{Z}$ be defined by
\begin{eqnarray}
\mu_{\mathbf{R}}(\mbox{\boldmath $a$})=\prod_{i=0}^{n-1}(a_{i+_n1}-_Na_i),
\end{eqnarray}
where the product on the right hand side is that of integers. For any marked rhythm $(k,\mbox{\boldmath $a$})\in m\mathbf{R}_N^n$, we put
\begin{eqnarray}
\mu_{m\mathbf{R}}(k,\mbox{\boldmath $a$})=\mu_{\mathbf{R}}(\mbox{\boldmath $a$}).
\end{eqnarray}
\end{df} 

\noindent
By the very definition, the function $\mu_{\mathbf{R}}$ satisfies the invariance property
\begin{eqnarray}
\mu_{\mathbf{R}}\circ tr = \mu_{\mathbf{R}},
\end{eqnarray}
since each factor on the right hand side of (3.1) depends only on the difference of the neighboring entries. Furthermore, it follows from (3.2) and (3.3) that
\begin{eqnarray}
\mu_{m\mathbf{R}}\circ mtr = \mu_{m\mathbf{R}}.
\end{eqnarray}
Hence we obtain the following.

\begin{prp}
The function $\mu_{m\mathbf{R}}$ on $m\mathbf{R}_N^n$ defines an invariant measure on the dynamical system $\mathbf{Ref}_N^n$ with respect to the automorphism $mtr$ of $\mathbf{Ref}_N^n$.
\end{prp}

By the invariance (3.3), $\mu_{\mathbf{R}}$ descends to the quotient $\mathbf{D}_N^n$. To be more specific, let $\mu_{\mathbf{D}}:\mathbf{D}_N^n\rightarrow\mathbf{Z}$ be defined by
\begin{eqnarray*}
\mu_{\mathbf{D}}(\mbox{\boldmath $d$})=\prod_{i=0}^{n-1}d_i
\end{eqnarray*}
for any $\mbox{\boldmath $d$}=(d_0,d_1,\cdots,d_{n-1})\in\mathbf{D}_N^n$. Then we have the commutative diagram:
\begin{eqnarray}
  \begin{CD}
     \mathbf{R}_N^n @>{\mu_{\mathbf{R}}}>> \mathbf{Z} \\
  @V{d}VV @V{id_{\mathbf{Z}}}VV \\
    \mathbf{D}_N^n @>{\mu_{\mathbf{D}}}>> \mathbf{Z} \\
  \end{CD}
\end{eqnarray}
\noindent
We extend the domain $\mathbf{D}_N^n$ of $\mu_{\mathbf{D}}$ to $m\mathbf{D}_N^n$ in the same way as above, and denote the extended measure by $\mu_{m\mathbf{D}}$. Namely, we put
\begin{eqnarray*}
\mu_{m\mathbf{D}}(k,\mbox{\boldmath $d$})=\mu_{\mathbf{D}}(\mbox{\boldmath $d$}).
\end{eqnarray*}
It follows that we have the commutative diagram.
\begin{eqnarray}
  \begin{CD}
     m\mathbf{R}_N^n @>{\mu_{m\mathbf{R}}}>> \mathbf{Z} \\
  @V{\Delta}VV @V{id_{\mathbf{Z}}}VV \\
    m\mathbf{D}_N^n @>{\mu_{m\mathbf{D}}}>> \mathbf{Z} \\
  \end{CD}
\end{eqnarray}

\noindent
Thus $\mu_{m\mathbf{D}}$ is identified with $\overline{\mu_{m\mathbf{R}}}$, the map on the quotient $m\mathbf{D}_N^n=m\mathbf{R}_N^n/\sim_{mtr}$ induce from $\mu_{m\mathbf{R}}$. Therefore the general result Proposition 2.2 gives us the following specific result.

\begin{prp}
For any $\mbox{\boldmath $A$}\in m\mathbf{R}_N^n$, it belongs to $Per(\mathbf{Ref})$ if and only if $\Delta(\mbox{\boldmath $A$})\in Per(\mathbf{Def})$.
\end{prp}

\noindent
Furthermore we can show the monotonicity of the measures $\mu_{m\mathbf{R}}$ and $\mu_{m\mathbf{D}}$.

\begin{prp}
${\rm (1)}$ For any $(k, \mbox{\boldmath $a$})\in m \mathbf{R}_N^n$, we have
\begin{eqnarray*}
\mu_{m\mathbf{R}}(Ref(k,\mbox{\boldmath $a$}))\geq \mu_{m\mathbf{R}}(k,\mbox{\boldmath $a$}).
\end{eqnarray*}
${\rm (2)}$ For any $(k,\mbox{\boldmath $d$})\in m\mathbf{D}_N^n$, we have
\begin{eqnarray}
\mu_{m\mathbf{D}}(Def(k,\mbox{\boldmath $d$}))\geq \mu_{m\mathbf{D}}(k,\mbox{\boldmath $d$}).
\end{eqnarray}
${\rm (3)}$ When $\mbox{\boldmath $d$}=(d_0,\cdots,d_{n-1})$, the equality in $(3.7)$ holds if and only if $|d_k-d_{k+_n1}|\leq 1$.
\end{prp}

\begin{proof}
By Proposition 2.2, the item (1) follows from the item (2). Hence we have only to prove the items (2) and (3). Furthermore the deformation map changes only an adjacent pair of the entries, we are reduced to showing the following:

\begin{lem}
For any pair $(a,b)$ of positive integers, we have
\begin{eqnarray}
a\cdot b\leq \left\lfloor\frac{a+b}{2}\right\rfloor\cdot \left\lceil\frac{a+b}{2}\right\rceil.
\end{eqnarray}
The equality holds if and only if $|a-b|\leq 1$.
\end{lem}

\noindent
We may assume that $a\leq b$. When $a=b$, both sides of (3.8) are equal to $a^2$, and hence the equality holds. When $b=a+1$, the right hand side of (3.8) is equal to $a\cdot (a+1)$, and hence the equality holds too. Suppose that $b> a+1$. In case $a\equiv b\pmod 2$, the right hand side of (3.8) is equal to $\displaystyle{\left(\frac{a+b}{2}\right)^2}$, and hence the difference (RHS)-(LHS) becomes $((a+b)^2-4ab)/4=(a-b)^2/4\geq 0$, which implies the validity of (3.8). In case $a\not\equiv b\pmod 2$, let $a+b=2k +1$ for an integer $k$. Then the right hand side is equal to $k(k+1)$. On the other hand, the left hand side becomes $a((2k+1)-a)$. Note that the quadratic function $x((2k+1)-x)$ on $\mathbf{Z}$ takes the maximum value at $x=k, k+1$. Since we are in the case that $b>a+1$, $a$ cannot equal to $k$ and hence the product $ab=a((2k+1)-a)$ is smaller than $k(k+1)$. This completes the proof.
\end{proof}

Combining Proposition 3.1 and 3.3, we arrive at the following objective of this section.

\begin{prp}
{\rm (1)} $\mu_{m\mathbf{R}}$ is a monotone invariant measure on the dynamical system $\mathbf{Ref}_N^n$ with respect to the automorphism $mtr$ on $\mathbf{Ref}_N^n$.\\
{\rm (2)} $\mu_{m\mathbf{D}}$ is a monotone measure on the dynamical system $\mathbf{Def}_N^n$.
\end{prp}

\section{Main theorem and its proof}
In this section we prove our main theorem which characterizes quasi-smooth rhythms in $m\mathbf{R}_N^n$. For our proof it is indispensable to focus on the structure of the set $St(\mu_{m\mathbf{D}})$ of $\mu_{m\mathbf{D}}$-stable points in the dynamical system $\mathbf{Def}_N^n=(m\mathbf{D}_N^n,Def)$. The following property is fundamental.

\begin{prp}
Assume that $\mbox{\boldmath $D$}=(k,\mbox{\boldmath $d$})$ is $(\mu_{m\mathbf{D}},Def)$-invariant, where $\mbox{\boldmath $d$}=(d_0,\cdots,d_{n-1})$. Then we have\\
$(1)$ $|d_k-d_{k+_n1}|\leq 1$.\\
$(2.1)$ When $d_{k+_n1}=d_k+1$, we have 
\begin{eqnarray}
Def(k,\mbox{\boldmath $d$})=(k+_n1,\mbox{\boldmath $d$}).
\end{eqnarray} 
$(2.2)$ When $d_{k+_n1}=d_k-1$ or $d_{k+_n1}=d_k$, we have 
\begin{eqnarray}
Def(k,\mbox{\boldmath $d$})=(k+_n1,ad_k(\mbox{\boldmath $d$})),
\end{eqnarray} 
where $ad_k:\mathbf{D}_N^n\rightarrow \mathbf{D}_N^n$ denotes the map which transposes the $k$-th and the $(k+_n1)$-th coordinates.
\end{prp}

\begin{proof}
Let $Def(k,\mbox{\boldmath $d$})=(k+_n1,\mbox{\boldmath $e$})$ with $\mbox{\boldmath $e$}=(e_0,\cdots,e_{n-1})$. The item (1) is already proved in Proposition 3.3 (3). When $d_{k+_n1}=d_k+1$, we have $(e_k,e_{k+_n1})=(d_k,d_{k+_n1})$ by the definition of the map $Def$, and hence (4.1) follows. When $d_{k+_n1}=d_k-1$, we have $e_k=\lfloor\frac{2d_k-1}{2}\rfloor=d_k-1=d_{k+_n1}$, and $e_{k+_n1}=\lceil\frac{2d_k-1}{2}\rceil=d_k$, and hence the equality (4.2) holds. Furthermore when $d_{k+_n1}=d_k$, the equality (4.2) holds trivially. This completes the proof.
\end{proof}

\noindent
For ease of description, we introduce the following notation.

\begin{df}
Let $\mbox{\boldmath $D$}=(k,\mbox{\boldmath $d$})\in\mathbf{D}_N^n$ with $\mbox{\boldmath $d$}=(d_0,\cdots,d_{n-1})$.\\
{\rm (1)} For any positive integer $\ell$, we express $Def^{\ell}(\mbox{\boldmath $D$})$ as
\begin{eqnarray*}
\mbox{\boldmath $D$}^{(\ell)}=(k+_n\ell,\mbox{\boldmath $d$}^{(\ell)})=(k+_n\ell,(d_0^{(\ell)},\cdots,d_{n-1}^{(\ell)})).
\end{eqnarray*}
{\rm (2)} The multiset of all the entries of $\mbox{\boldmath $d$}$ is called the {\it content} of $\mbox{\boldmath $D$}$, and is denoted by $Cont(\mbox{\boldmath $D$})$. 
\end{df}

\noindent
As a direct consequence of Proposition 4.1, we obtain the following.

\begin{cor}
If $\mbox{\boldmath $D$}\in St(\mu_{m\mathbf{D}})$, then for any positive integer $\ell$ we have
\begin{eqnarray*}
Cont(\mbox{\boldmath $D$}^{(\ell)})=Cont(\mbox{\boldmath $D$}).
\end{eqnarray*}
\end{cor}

Now we can prove the following result which plays a crucial role for our characterization of the quasi-smoothness. Recall that the difference of the maximum and the minimum of the entries of $\mbox{\boldmath $d$}\in\mathbf{D}_N^n$ is called the {\it width} of $\mbox{\boldmath $d$}$, and is denoted by $w(\mbox{\boldmath $d$})$. For any $\mbox{\boldmath $D$}=(k,\mbox{\boldmath $d$})\in m\mathbf{D}_N^n$ we set $w(\mbox{\boldmath $D$})=w(\mbox{\boldmath $d$})$, and call it the width of $\mbox{\boldmath $D$}$ too.

\begin{prp}
If $\mbox{\boldmath $D$}\in St(\mu_{m\mathbf{D}})$, then we have
\begin{eqnarray*}
w(\mbox{\boldmath $D$})\leq 1.
\end{eqnarray*}
\end{prp}

\begin{proof}
We may assume that the marker of $\mbox{\boldmath $D$}$ is equal to 0, and hence we set $\mbox{\boldmath $D$}=(0,\mbox{\boldmath $d$})$ with $\mbox{\boldmath $d$}=(d_0,\cdots,d_{n-1})$. Let $M=\max\{d_i|i\in\mathbf{Z}_n\}$ and let $i_m=\min\{i|d_i=M\}$, the first index of the entry which attains the maximum $M$. Then it follows from Proposition 4.1 (2.2) that, for any $\ell\geq i_m$, the marked entry of $\mbox{\boldmath $D$}^{(\ell)}$ is equal to $M$, and the subsequent entry belongs to $\{M, M-1\}$. Applying $Def$ repeatedly, the same argument shows that $Cont(\mbox{\boldmath $D$}^{(i_m+n-1)})\subset \{M,M-1\}$. Therefore it follows from Corollary 4.1 that $Cont(\mbox{\boldmath $D$})\subset \{M,M-1\}$, and hence we have $w(\mbox{\boldmath $D$})\leq 1$. This completes the proof.
\end{proof}

In order to characterize the set $Per(\mathbf{Def}_N^n)$ in $m\mathbf{D}_N^n$, we need to specify the value which the marked entry takes.

\begin{df}
For any $\mbox{\boldmath $D$}\in m\mathbf{D}_N^n$, we denote the maximum of its difference part by $\max(\mbox{\boldmath $D$})$. We say that $\mbox{\boldmath $D$}$ is {\rm max-marked}, if the value of the marked entry is equal to $\max(\mbox{\boldmath $D$})$. 
\end{df}

The following lemma enables us to characterize the subset $Per(\mathbf{Def}_N^n)$ in $m\mathbf{D}_N^n$.

\begin{lem} 
Suppose that $\mbox{\boldmath $D$}\in m\mathbf{D}_N^n$ has width $\leq 1$. \\
$(1)$ If it is max-marked, then for any positive integer $\ell$, the $\ell$-th deformation $\mbox{\boldmath $D$}^{(\ell)}$ is max-marked.\\
$(2)$ If it is \underline{not} max-marked, it is not periodic under the deformation map $Def$.
\end{lem}

\begin{proof}
(1) By the assumption that $\mbox{\boldmath $D$}$ is max-marked, the alternative (2.2) in Proposition 4.1 holds, and as a result $\mbox{\boldmath $D$}^{(1)}$ is max-marked. Repeating this argument we see that every $\mbox{\boldmath $D$}^{(\ell)}$ is max-marked for any positive integer $\ell$. \\
(2) Let $\mbox{\boldmath $D$}=(k, \mbox{\boldmath $d$})$ with $\mbox{\boldmath $d$}=(d_0,\cdots,d_{n-1})$, and $M=max(\mbox{\boldmath $D$})$. It follows from the assumption that $d_k=M-1$. Hence there exists an index $\ell\in\mathbf{Z}_n$ such that $d_{\ell}=M$ and that $d_m=M-1$ for every $m\in [k,\ell-_n1]$. Then it follows from Proposition 4.1 that $\mbox{\boldmath $D$}^{(\ell-_nk)}$ is max-marked. This implies by (1) that $\mbox{\boldmath $D$}^{m)}$ is max-marked for any $m\geq \ell-_nk$, and hence the original $\mbox{\boldmath $D$}$, which is not max-marked, cannot belong to a cycle. This completes the proof.
\end{proof}

\begin{thm}
${\rm (A)}$ A marked difference $\mbox{\boldmath $D$}\in m\mathbf{D}_N^n$ belongs to $Per(\mathbf{Def}_N^n)$ if and only if it satisfies the following two conditions.\\
$(A.1)$ $w(\mbox{\boldmath $D$})\leq 1$.\\
$(A.2)$ It is max-marked.\\
${\rm (B)}$ A marked rhythm $\mbox{\boldmath $A$}\in m\mathbf{R}_N^n$ is quasi-smooth if and only if its difference $\Delta(\mbox{\boldmath $A$})$ satisfies the above two conditions $(A.1)$ and $(A.2)$.
\end{thm}

\begin{proof}
(A) Only-if-part: Suppose that $\mbox{\boldmath $D$}\in Per(\mathbf{Def}_N^n)$. Since $\mu_{m\mathbf{D}}$ is a monotone measure by Proposition 3.3 (2), it follows from Corollary 2.1 that $\mbox{\boldmath $D$}$ is $\mu_{m\mathbf{D}}$-stable, which implies by Proposition 4.2 that $w(\mbox{\boldmath $D$})\leq 1$. Hence the icondition (A.1) holds. The validity of the condition (A.2) is a direct consequence of Lemma 4.1 (2).\\
If-part: Let $\mbox{\boldmath $D$}=(k, \mbox{\boldmath $d$})$ with $\mbox{\boldmath $d$}=(d_0,\cdots,d_{n-1})$. Under the conditions (A.1) and (A.2), we have
\begin{eqnarray*}
\mbox{\boldmath $D$}^{(1)}=(k+_n1,ad_k(\mbox{\boldmath $d$}))
\end{eqnarray*}
by Proposition 4.1 (2.2). Since $\mbox{\boldmath $D$}^{(1)}$ is max-marked itself, we have
\begin{eqnarray*}
\mbox{\boldmath $D$}^{(2)}=(k+_n2,ad_{k+_n1}ad_k(\mbox{\boldmath $d$})).
\end{eqnarray*}
By repeating these processes, we see that
\begin{eqnarray*}
\mbox{\boldmath $D$}^{(n-1)}=(k+_n(n-1),ad_{k+_n(n-2)}\cdots ad_{k+_n1}ad_k(\mbox{\boldmath $d$})).
\end{eqnarray*}
Notice here that, as the product of adjacent transpositions on $\mathbf{Z}_n$, we have the equality 
\begin{eqnarray*}
(n-2\hspace{2mm}n-1)\cdots (1\hspace{2mm}2)(0\hspace{2mm}1)=(n-1\hspace{2mm}n-2\hspace{2mm}\cdots\hspace{2mm}2\hspace{2mm}1 \hspace{2mm}0),
\end{eqnarray*}
and a slightly general one
\begin{eqnarray*}
&&(k+_n(n-2)\hspace{2mm}k+_n(n-1))\cdots (k+_n1\hspace{2mm}k+_n2)(k\hspace{2mm}k+_n1)\\
&&=(k+_n(n-1)\hspace{2mm}k+_n(n-2)\hspace{2mm}\cdots \hspace{2mm}k+_n2\hspace{2mm}k+_n1 \hspace{2mm}k)\\
&&=(n-1\hspace{2mm}n-2\hspace{2mm}\cdots\hspace{2mm}2\hspace{2mm}1 \hspace{2mm}0),
\end{eqnarray*}
the last equality coming from the fact that both of the last two cyclic permutations map $m$ to $m-_n1$ for any $m\in\mathbf{Z}_n$. Since the cycle $(n-1\hspace{2mm}n-2\hspace{2mm}\cdots\hspace{2mm}2\hspace{2mm}1 \hspace{2mm}0)$ is of order $n$, we see that
\begin{eqnarray*}
\mbox{\boldmath $D$}^{(n(n-1))}=\mbox{\boldmath $D$},
\end{eqnarray*}
and hence $\mbox{\boldmath $D$}\in Per(\mathbf{Def}_N^n)$. \\
(B) The assertion is a direct consequence of Proposition 3.2 which establishes the equivalence of the periodicity of $\mbox{\boldmath $A$}$ and that of $\Delta(\mbox{\boldmath $A$})$. This completes the proof.
\end{proof}

Thus we arrive at the main theorem of this paper.

\begin{thm}
A rhythm is smooth if and only if it is quasi-smooth. In particular, for any rhythm $\mbox{\boldmath $a$}\in \mathbf{R}_N^n$, there exists a nonnegative integer $\ell$ such that $p_2(Ref^{\ell}(\iota_0(\mbox{\boldmath $a$})))$ is smooth, where $p_2:m\mathbf{R}_N^n\rightarrow \mathbf{R}_N^n$ denotes the projection onto the second factor.
\end{thm}

\begin{proof}
Recall that, by the main theorem of (Hazama 2022), a rhythm $\mbox{\boldmath $a$}$ is smooth if and only if 
\begin{eqnarray}
w(d(\mbox{\boldmath $a$}))\leq 1. 
\end{eqnarray}
Therefore the if-part follows from Theorem 4.1 (B). As for the only-if part, assume that $\mbox{\boldmath $a$}$ is smooth. Let $\mbox{\boldmath $d$}=d(\mbox{\boldmath $a$})=(d_0,\cdots, d_{n-1})$, and choose an index $k\in\mathbf{Z}_n$ such that $d_k=\max(\mbox{\boldmath $d$})$. Then $\mbox{\boldmath $D$}=\iota_k(\mbox{\boldmath $d$})\in m\mathbf{D}_N^n$ is max-marked and satisfies the condition $w(\mbox{\boldmath $D$})\leq 1$ by (4.3). Therefore Theorem 4.1 (B) implies that $\iota_k(\mbox{\boldmath $a$})$ is quasi-smooth, and hence $\mbox{\boldmath $a$}$ is quasi-smooth.\\
\end{proof}

\newpage
\noindent
{\bf References}\\\\
Clough J, Douthett J, Krantz R (2000), Maximally Even Sets: A Discovery in Mathematical Music Theory is
Found to Apply in Physics. In: Reza S(ed) Bridges: Mathematical Connections in Art, Music, and Science, Conference Proceedings 2000. Central Plain Book Manufacturing, Kansas, pp 193-200. \\\\
Demaine ED, Gomez-Martin F, Meijer H, Rappaport D, Taslakian P, Toussaint  GT, Winograd T, Wood  DR (2009), The distance geometry of music. Computational Geometry 42: 429-454. \\\\
Hazama F  (2022) Iterative method of construction for smooth rhythms. Journal of Mathematics and Music 16:216-235.\\\\
Rota GC (2001) The adventures of measure theory. In Crapo H, Senato D(eds) Algebraic Combinatorics and Computer Science. Springer-Verlag Italia, pp 27-39.\\\\
Toussaint. G (2013) The Geometry of Musical Rhythm: What Makes a ``Good" Rhythm Good?  CRC Press, New York.\\\\

\end{document}